\pdfoutput=1
\documentclass[11pt]{article}

\usepackage[a4paper,margin=1in]{geometry}
\usepackage[T1]{fontenc}
\usepackage[utf8]{inputenc}
\usepackage{lmodern}
\usepackage{microtype}
\usepackage{amsmath,amssymb,amsthm}
\usepackage[numbers,sort&compress]{natbib}
\bibpunct[, ]{[}{]}{,}{n}{,}{,}
\makeatletter
\def\NAT@def@citea{\def\@citea{\NAT@separator}}
\makeatother
\usepackage{longtable}
\usepackage{array}
\usepackage{booktabs}
\usepackage{enumitem}
\usepackage{tikz}
\usepackage{float}
\usetikzlibrary{arrows.meta,positioning}
\usepackage[hidelinks]{hyperref}

\newcommand{\SV}{\operatorname{SV}}
\newcommand{\Pow}{\mathcal{P}}

\theoremstyle{plain}
\newtheorem{theorem}{Theorem}[section]
\newtheorem{lemma}[theorem]{Lemma}

\newtheorem{proposition}[theorem]{Proposition}

\theoremstyle{definition}
\newtheorem{definition}[theorem]{Definition}
\newtheorem{example}[theorem]{Example}

\theoremstyle{remark}
\newtheorem{remark}[theorem]{Remark}

\title{Scale-valued sets: a minimal framework for generalized set models}
\author{Subhasis Ray\\
Department of Mathematics, Visva-Bharati University, Santiniketan-731235, India\\
\texttt{subhasis.ray@visva-bharati.ac.in}}
\date{}

\begin{document}

\maketitle

\begin{abstract}
Many generalized set models have the same basic form: they assign a value to each object, and the main difference lies in the kind of values that are allowed. This paper studies that common form through scale-valued sets (SV-sets), defined as maps $U\times E\to\Sigma$, where $U$ is a universe, $E$ is a parameter set, and $\Sigma$ is a bounded De Morgan lattice. With a suitable choice of scale, SV-sets include ordinary sets, fuzzy sets, soft sets, bounded multisets, intuitionistic fuzzy sets, $L$-fuzzy sets, and Type-2 fuzzy sets. We study the basic structure of SV-sets. The relation between SV-sets and lattice-valued interval soft sets is also discussed. For complete chains, the SV setting gives a natural topological construction, and for groups, it gives an algebraic structure through SV-subgroups. The applications show how graded suitability and supporting evidence can be kept together in a single model, whereas one-coordinate reductions lose information.
\end{abstract}


\section{Introduction}

Generalized set models were introduced to describe kinds of information that ordinary sets do not record well, such as graded membership, parameter dependence, multiplicity, or approximation. Fuzzy sets~\cite{zadeh1965}, soft sets~\cite{molodtsov1999}, multisets~\cite{hickman1980}, rough sets~\cite{pawlak1982}, intuitionistic fuzzy sets~\cite{atanassov1986}, and $L$-fuzzy sets~\cite{goguen1967} are familiar examples. Although these theories come from different motivations, many of them share the same outer form: they assign a value to each object, and in some cases that assignment is indexed by parameters.

This paper studies that common form directly. A \emph{scale-valued set} (SV-set) is a map
\[
A:U\times E\to\Sigma,
\]
where $U$ is a universe, $E$ is a parameter set, and $\Sigma$ is a bounded De Morgan lattice, called a \emph{scale}. $\Sigma$ records the kind of local membership information one wants to store. By changing the scale, one moves from Boolean membership to fuzzy grades, multiplicities, paired values, and other local descriptions without changing the outer definition.

 What the SV setting offers is a simple outer layer in which common structural facts can be stated once, and hybrid local information can be handled by choosing an appropriate scale.
After the basic notion, the pointwise properties of SV-sets and the natural maps between them are studied. Several standard models, including Type-2 fuzzy sets, are then recovered within this setting, and the relation with lattice-valued interval soft sets is also explained. After that, the notions of SV-topological spaces and SV-subgroups are presented, and product-scale applications are developed in which graded suitability and supporting evidence are kept together in a single scale.

\section{Definition of an SV-set and fundamental operations}\label{sec:basic-notion}

\begin{definition}\label{def:scale}
A \emph{scale} is a structure $\Sigma=(\Sigma,\vee,\wedge,0,1,\neg)$ satisfying the following conditions.
\begin{enumerate}[label=(\roman*)]
    \item $(\Sigma,\vee,\wedge)$ is a lattice.
    \item $0$ is the least element and $1$ is the greatest element of $\Sigma$.
    \item $\neg\neg a=a$ for every $a\in\Sigma$.
    \item $a\leq b$ implies $\neg b\leq \neg a$.
    \item $\neg(a\vee b)=\neg a\wedge \neg b$ and $\neg(a\wedge b)=\neg a\vee \neg b$ for all $a,b\in\Sigma$.
\end{enumerate}
Thus a scale is a bounded De Morgan lattice. The order $\leq$ is the lattice order, equivalently $a\leq b$ if and only if $a\wedge b=a$.
\end{definition}

\begin{definition}\label{def:svset}
Let $U$ be a nonempty universe and let $E$ be any set of parameters or contexts. An SV-set over $U$ with parameter set $E$ and scale $\Sigma$ is a map
\;
$A: U\times E \to \Sigma.$\smallskip

\noindent  We use $E=\{\ast\}$ to denote the function $A:U\to\Sigma$ and call it an \emph{unparameterized} SV-set.\smallskip

\noindent If $A$ is an SV-set and $e\in E$, the $e$-slice of $A$ is the map $A_e:U\to\Sigma$ defined by $A_e(x)=A(x,e)$.\smallskip

\noindent The value $A(x,e)$ is interpreted as the membership grade of $x$ under the parameter  $e$. In this way the parameter set $E$ plays the role familiar from soft set theory, while the scale $\Sigma$ controls the type of grades that are allowed.\smallskip

\noindent The collection
$\SV(U,E;\Sigma)=\{A\mid A:U\times E\to\Sigma\}$ represents
all SV-sets over $U$.
\end{definition}

\begin{example}
Let $U=\{u_1,u_2,u_3\}$, $E=\{p,q\}$, and $\Sigma=\{0,1,2\}$. A map $A:U\times E\to\Sigma$ may assign the values $A(u_1,p)=2$, $A(u_2,p)=1$, $A(u_3,p)=0$, $A(u_1,q)=1$, $A(u_2,q)=0$, and $A(u_3,q)=2$. Then the slices $A_p$ and $A_q$ are two chain-valued membership functions on the same universe. If the scale is changed from $\{0,1,2\}$ to $\{0,1\}$, the same outer definition becomes a soft-set style Boolean relation.
\end{example}

\begin{definition}
For $A,B\in \SV(U,E;\Sigma)$, define union, intersection, and complement pointwise by
\[
(A\vee B)(x,e)=A(x,e)\vee B(x,e),
\quad
(A\wedge B)(x,e)=A(x,e)\wedge B(x,e),
\quad
A^c(x,e)=\neg A(x,e).
\]
We write $A\subseteq B$ when $A(x,e)\leq B(x,e)$ for all $(x,e)\in U\times E$, and $A=B$ when the two functions are equal. 

For fixed $U$, $E$, and $\Sigma$, the class $\SV(U,E;\Sigma)$ is a bounded De Morgan lattice as $\Sigma$ is so.
Similarly if the scale $\Sigma$ is a complete lattice, then $\SV(U,E;\Sigma)$ is also complete.
\end{definition}

\begin{definition}
Let $\Sigma$ and $T$ be scales. A map $h:\Sigma\to T$ is called a \emph{scale homomorphism} if it preserves $\vee$, $\wedge$, $0$, $1$, and $\neg$.
\end{definition}

\begin{theorem}
Let $h:\Sigma\to T$ be a scale homomorphism. Define $h_*A\in \SV(U,E;T)$ by $(h_*A)(x,e)=h(A(x,e))$. Then $h_*$ preserves union, intersection, complement, and the inclusion relation.
\end{theorem}

\begin{proof}
For union, 
\[
(h_*(A\vee B))(x,e)=h(A(x,e)\vee B(x,e))=h(A(x,e))\vee h(B(x,e))=(h_*A\vee h_*B)(x,e).
\]
The proofs for intersection and complement are similar. If $A\subseteq B$, then $A(x,e)\leq B(x,e)$ for all $(x,e)$, and as the lattice homomorphisms are order-preserving we get $h(A(x,e))\leq h(B(x,e))$; hence $h_*A\subseteq h_*B$.
\end{proof}

\begin{theorem}
Let $f:U'\to U$ and $g:E'\to E$ be maps. Define $(f,g)^*A\in \SV(U',E';\Sigma)$ by $(f,g)^*A=A\circ (f\times g)$. Then $(f,g)^*$ preserves union, intersection, complement, and the subset order.
\end{theorem}

\begin{proof}
For all $(x',e')\in U'\times E'$, one has
\[
\begin{aligned}
((f,g)^*(A\vee B))(x',e')
&=(A\vee B)(f(x'),g(e'))\\
&=A(f(x'),g(e'))\vee B(f(x'),g(e'))\\
&=((f,g)^*A\vee (f,g)^*B)(x',e').
\end{aligned}
\]
The remaining identities can be proved in the same way. If $A\subseteq B$, then $A(f(x'),g(e'))\leq B(f(x'),g(e'))$ for all $(x',e')$, so $(f,g)^*A\subseteq (f,g)^*B$.
\end{proof}

\begin{definition}\label{def 2.8}
Assume that $\Sigma$ is complete and let $f:U\to V$ be a map. For $A\in \SV(U,E;\Sigma)$ define $f_!A\in \SV(V,E;\Sigma)$ by
\[
(f_!A)(v,e)=\bigvee_{x\in U\,:\, f(x)=v} A(x,e),
\]
where the empty join is interpreted as $0$.
\end{definition}

\begin{proposition}
If $\Sigma$ is complete, then $f_!$ is monotone and preserves arbitrary pointwise unions:
\[
f_!\left(\bigvee_{i\in I}A_i\right)=\bigvee_{i\in I}f_!(A_i).
\]
If, in addition, $\Sigma$ is a chain, then $f_!A(v,e)$ is the supremum of the grades on the preimage $f^{-1}(v)$. In particular, if the supremum is attained (for example, in a finite chain or on a finite preimage), then $f_!A(v,e)$ is the maximum grade on that preimage.
\end{proposition}

\begin{proof}
Monotonicity follows because joins are monotone in each variable. For arbitrary joins,
\[
\begin{aligned}
\left(f_!\left(\bigvee_{i\in I}A_i\right)\right)(v,e)
&=\bigvee_{x\,:\, f(x)=v}\left(\bigvee_{i\in I}A_i(x,e)\right)\\
&=\bigvee_{i\in I}\bigvee_{x\,:\, f(x)=v}A_i(x,e)
=\left(\bigvee_{i\in I}f_!(A_i)\right)(v,e).
\end{aligned}
\]
If $\Sigma$ is a chain, the join over a preimage is simply its supremum, that is, the maximum value when the supremum is attained, as happens for finite preimages or finite chains.
\end{proof}

\section{Generalized-Set Models as Special Cases of SV-Sets and Related Encodings}
The familiar sets ~\cite{zadeh1965,molodtsov1999,hickman1980,pawlak1982,atanassov1986,goguen1967} can be viewed as special cases of SV-set by suitable choices of the parameter and scale.
\begin{theorem}\label{thm main}
Fix a universe $U$. Then the following statements hold.
\begin{enumerate}[label=(\roman*)]
    \item Classical sets (crisp sets) $S\subseteq U$ are SV-sets with $E=\{\ast\}$ and $\Sigma=\{0,1\}$.

    \item Fuzzy sets on $U$ are SV-sets with $E=\{\ast\}$ and $\Sigma=[0,1]$ equipped with $\vee=\max$, $\wedge=\min$, and $\neg t=1-t$.

    \item Soft sets on $U$ with parameter set $E$ are SV-sets $A:U\times E\to\{0,1\}$.

    \item For $k\in\mathbb{N}$, bounded multisets with multiplicities in $\{0,1,\dots,k\}$ are SV-sets with $E=\{\ast\}$ and $\Sigma=\{0,1,\dots,k\}$ equipped with $\vee=\max$, $\wedge=\min$, and $\neg n=k-n$.

    \item Intuitionistic fuzzy sets on $U$ are SV-sets with $E=\{\ast\}$ and the scale
    \[
    \Delta=\{(a,b)\in[0,1]^2:\ a+b\le 1\},
    \]
    ordered by $(a,b)\le(c,d)$ if and only if $a\le c$ and $b\ge d$, with
    \[
    \begin{aligned}
    (a,b)\vee(c,d)&=(\max(a,c),\min(b,d)),\\
    (a,b)\wedge(c,d)&=(\min(a,c),\max(b,d)),\\
    \neg(a,b)&=(b,a).
    \end{aligned}
    \]

    \item If $L$ is taken as a bounded De Morgan lattice, then $L$-fuzzy sets $\mu:U\to L$ are SV-sets with $E=\{\ast\}$ and $\Sigma=L$. If $L$ is only a bounded lattice, the correspondence holds for joins and meets, but not necessarily for complements.

    \item If a rough set is represented abstractly by a pair of sets $L\subseteq M\subseteq U$ (for example, the lower and upper approximations $L(X)\subseteq U(X)$ of some $X\subseteq U$), then it determines an SV-set $A:U\to R$ by
    \[
    R=\{(0,0),(0,1),(1,1)\},\qquad A(x)=\bigl(\chi_L(x),\chi_M(x)\bigr).
    \]
    Moreover, with the chain order $(0,0)<(0,1)<(1,1)$, $\vee=\max$, $\wedge=\min$, and
    \[
    \neg(0,0)=(1,1),\qquad \neg(1,1)=(0,0),\qquad \neg(0,1)=(0,1),
    \]
    the map $(L,M)\mapsto A$ is a bijection between pairs $L\subseteq M\subseteq U$ and SV-sets $U\to R$, and it preserves union, intersection, and complement at the level of pairs.

    \item General Type-2 fuzzy sets on $U$ are SV-sets with $E=\{\ast\}$ and scale
    $\Sigma_{\mathrm{T2}}=[0,1]^{[0,1]}$
    equipped with the pointwise order, pointwise joins and meets, bounds $0(u)=0$, $1(u)=1$, and pointwise negation $(\neg f)(u)=1-f(u)$.
\end{enumerate}
\end{theorem}

\begin{proof}
\begin{enumerate}[label=(\roman*)]
\item
Given $S\subseteq U$, define $A_S:U\times\{\ast\}\to\{0,1\}$ by $A_S(x,\ast)=\chi_S(x)$. Conversely, given an SV-set $A:U\times\{\ast\}\to\{0,1\}$, define $S_A=\{x\in U:A(x,\ast)=1\}$. Clearly $\chi_{S_A}(x)=A(x,\ast)$ for all $x\in U$.
\smallskip
\item
A fuzzy set on $U$ is a map $\mu:U\to[0,1]$~\cite{zadeh1965}. Define the SV-set $A_\mu:U\times\{\ast\}\to[0,1]$ by $A_\mu(x,\ast)=\mu(x)$, and conversely define $\mu_A(x)=A(x,\ast)$.
\smallskip
\item
A soft set on $U$ with parameters $E$ is presented as a map $F:E\to\Pow(U)$~\cite{molodtsov1999}. Given $F$, define $A_F:U\times E\to\{0,1\}$ by $A_F(x,e)=1$ if $x\in F(e)$. Conversely, given $A:U\times E\to\{0,1\}$, define $F_A(e)=\{x\in U:A(x,e)=1\}$. These constructions are inverse since $x\in F_A(e)$ if and only if $A(x,e)=1$. Soft union and intersection are parameterwise,
\[
(F\cup G)(e)=F(e)\cup G(e),\qquad (F\cap G)(e)=F(e)\cap G(e),
\]
so their characteristic functions are pointwise max and min. Hence $A_{F\cup G}=A_F\vee A_G$ and $A_{F\cap G}=A_F\wedge A_G$. Complements, when used, are likewise parameterwise and correspond to SV-complement.
\smallskip
\item
A bounded multiset with bound $k$ is presented as a multiplicity function $m:U\to\{0,1,\dots,k\}$~\cite{hickman1980}. Define $A_m:U\times\{\ast\}\to\{0,1,\dots,k\}$ by $A_m(x,\ast)=m(x)$, with inverse $m_A(x)=A(x,\ast)$. This correspondence preserves joins, meets, and complement.
\smallskip
\item
An intuitionistic fuzzy set is a pair $(\mu,\nu)$ of functions $\mu,\nu:U\to[0,1]$ with $\mu(x)+\nu(x)\le 1$ for all $x\in U$~\cite{atanassov1986}. Define $A_{\mu,\nu}:U\times\{\ast\}\to\Delta$ by $A_{\mu,\nu}(x,\ast)=(\mu(x),\nu(x))$, and conversely let $(\mu_A(x),\nu_A(x))=A(x,\ast)$. This is a bijection because the pointwise constraint $\mu+\nu\le 1$ is exactly the condition that $A(x,\ast)\in\Delta$. The listed formulas on $\Delta$ are precisely the usual intuitionistic-fuzzy union, intersection, and complement operations written pointwise, so the correspondence preserves joins, meets, and complement.
\smallskip
\item
An $L$-fuzzy set is a map $\mu:U\to L$ for an appropriate truth-value structure $L$~\cite{goguen1967}. If $L$ is taken as a bounded De Morgan lattice, then the same identity construction as in part~(ii) gives a bijection between $L$-fuzzy sets and SV-sets $U\times\{\ast\}\to L$, and the pointwise SV-operations preserve the lattice operations in $L$. If $L$ is only a bounded lattice, the construction still preserves joins and meets, but complements need not be available.
\smallskip
\item
In Pawlak rough set theory, a rough set can be represented by a pair of approximations $(L(X),M(X))$ with $L(X)\subseteq M(X)\subseteq U$~\cite{pawlak1982}. Consider the class of all pairs $(L,M)$ with $L\subseteq M\subseteq U$. Define $A_{L,M}:U\to R$ by $A_{L,M}(x)=(\chi_L(x),\chi_M(x))$. Conversely, given $A:U\to R$, define
\[
L_A=\{x\in U:A(x)=(1,1)\},\qquad M_A=\{x\in U:A(x)\neq(0,0)\}.
\]
Then $L_A\subseteq M_A$ holds automatically, and $A_{L_A,M_A}=A$ on the three values in $R$; also $L_{A_{L,M}}=L$ and $M_{A_{L,M}}=M$. Thus $(L,M)\leftrightarrow A$ is a bijection between pairs $L\subseteq M\subseteq U$ and $R$-valued SV-sets. For operations, define on pairs
\[
(L,M)\cup(L',M')=(L\cup L',M\cup M'),\qquad (L,M)\cap(L',M')=(L\cap L',M\cap M'),
\]
and complement by
\[
(L,M)^c=(U\setminus M, U\setminus L).
\]
Since $\chi_{A\cup B}=\max(\chi_A,\chi_B)$ and $\chi_{A\cap B}=\min(\chi_A,\chi_B)$ pointwise, and since $\neg(\chi_L,\chi_M)=(1-\chi_M,1-\chi_L)$ corresponds to $(U\setminus M,U\setminus L)$, the bijection preserves join, meet, and complement induced by the De Morgan structure on $R$.
\smallskip
\item
A general Type-2 fuzzy set~\cite{mendeljohn2002} may be written as a membership function
\[
\mu_{\widetilde A}:U\times[0,1]\to[0,1],
\]
where, for each fixed $x\in U$, the slice $u\mapsto \mu_{\widetilde A}(x,u)$ is the secondary membership function of $x$ (with the convention that it is $0$ outside its primary domain $J_x$, if variable primary domains are used). Since $[0,1]$ is a complete chain, its function space $[0,1]^{[0,1]}$ inherits pointwise suprema, infima, bounds, and De Morgan negation. Hence $\Sigma_{\mathrm{T2}}$ is a complete scale. Now define
\[
A^{\mathrm{T2}}(x,\ast)=f_x,\qquad f_x(u)=\mu_{\widetilde A}(x,u).
\]
Then $A^{\mathrm{T2}}(x,\ast)\in\Sigma_{\mathrm{T2}}$ for every $x\in U$, so $A^{\mathrm{T2}}$ is an SV-set.\\ Conversely, if
\[
A:U\times\{\ast\}\to\Sigma_{\mathrm{T2}}, \quad
\text{
define}
\quad
\mu_{\widetilde A}(x,u)=A(x,\ast)(u).
\]
These two constructions are inverse to one another, so general Type-2 fuzzy sets on $U$ correspond bijectively to SV-sets with $E=\{\ast\}$ and scale $\Sigma_{\mathrm{T2}}$.
\end{enumerate}
\end{proof}

\begin{remark}
When only the footprint of uncertainty is needed, an Interval Type-2 fuzzy set may be packaged by the interval scale
\[
\Sigma_{\mathrm{IT2}}=\mathcal{I}([0,1])=\{[l,u]\subseteq[0,1]:0\le l\le u\le 1\} \quad
\text{
via}
\quad
A_{\mathrm{IT2}}(x,\ast)=\bigl[\underline{\mu}_{\widetilde A}(x),\overline{\mu}_{\widetilde A}(x)\bigr].
\]
This gives the familiar lower/upper membership inside the SV-set setting. 

If one also wants an external parameter set $E$, the same construction gives
\[
A:U\times E\to\Sigma_{\mathrm{T2}},\qquad A(x,e)(u)=\mu_{\widetilde A}(x,e,u).
\]
Thus the SV parameter set keeps track of context, labels, or criteria separately from the primary and secondary membership structure used in Type-2 fuzziness.
\end{remark}

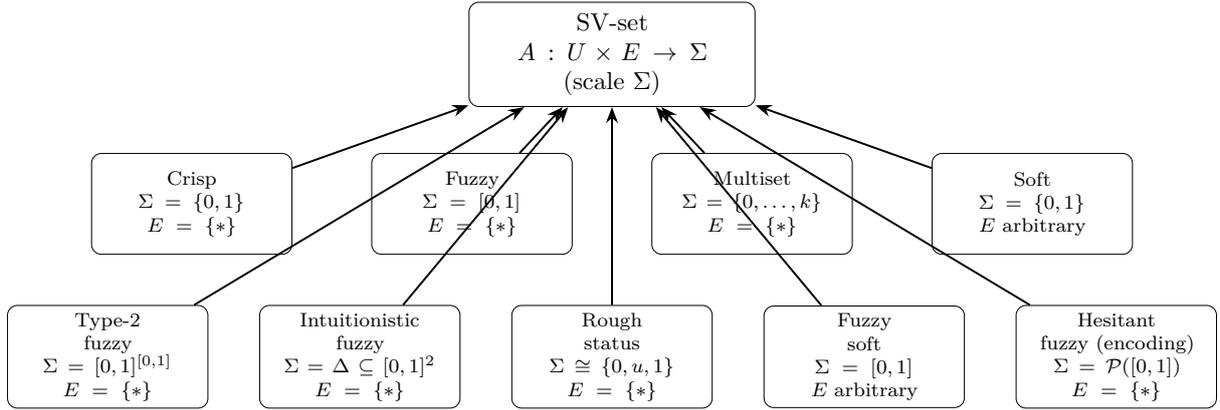
\begin{figure}[H]
\centering
\resizebox{\textwidth}{!}{%
\begin{tikzpicture}[
  svbox/.style={draw, rounded corners, align=center, inner sep=4pt, text width=3.8cm, minimum height=1.5cm, font=\small},
  box/.style={draw, rounded corners, align=center, inner sep=3pt, text width=2.65cm, minimum height=1.45cm, font=\scriptsize},
  arr/.style={-Stealth, thick}
]
\node[svbox] (sv) at (0,3.2) {SV-set\\$A:U\times E\to\Sigma$\\(scale $\Sigma$)};

\node[box] (crisp) at (-6.0,1.05) {Crisp\\$\Sigma=\{0,1\}$\\$E=\{\ast\}$};
\node[box] (fuzzy) at (-2.0,1.05) {Fuzzy\\$\Sigma=[0,1]$\\$E=\{\ast\}$};
\node[box] (multi) at ( 2.0,1.05) {Multiset\\$\Sigma=\{0,\dots,k\}$\\$E=\{\ast\}$};
\node[box] (soft)  at ( 6.0,1.05) {Soft\\$\Sigma=\{0,1\}$\\$E$ arbitrary};

\node[box] (type2) at (-7.2,-1.15) {Type-2\\fuzzy\\$\Sigma=[0,1]^{[0,1]}$\\$E=\{\ast\}$};
\node[box] (ifs)   at (-3.6,-1.15) {Intuitionistic\\fuzzy\\$\Sigma=\Delta\subseteq[0,1]^2$\\$E=\{\ast\}$};
\node[box] (rough) at ( 0.0,-1.15) {Rough\\status\\$\Sigma\cong\{0,u,1\}$\\$E=\{\ast\}$};
\node[box] (fsoft) at ( 3.6,-1.15) {Fuzzy\\soft\\$\Sigma=[0,1]$\\$E$ arbitrary};
\node[box] (hes)   at ( 7.2,-1.15) {Hesitant\\fuzzy (encoding)\\$\Sigma=\mathcal P([0,1])$\\$E=\{\ast\}$};

\draw[arr] (crisp) -- (sv);
\draw[arr] (fuzzy) -- (sv);
\draw[arr] (multi) -- (sv);
\draw[arr] (soft) -- (sv);
\draw[arr] (type2) -- (sv);
\draw[arr] (ifs) -- (sv);
\draw[arr] (rough) -- (sv);
\draw[arr] (fsoft) -- (sv);
\draw[arr] (hes) -- (sv);
\end{tikzpicture}%
}
\caption{SV-sets as a single umbrella definition: many generalized-set models arise by choosing the scale $\Sigma$ and the parameter set $E$. The hesitant fuzzy case is shown here at the level of encoding; see Torra~\cite{torra2010}.}
\label{fig:umbrella}
\end{figure}

\begin{remark}
 For the results above, only a small amount of structure is  needed on the scale $\Sigma$. To handle unions and intersections pointwise, it is enough that $\Sigma$ has bottom and top elements together with joins and meets. To handle complements pointwise, it needs an order-reversing involution satisfying the De Morgan laws.  if a model does not have a natural complement, a bounded lattice already suffices. If one wants to go further and include fuzzy-logical operations such as non-idempotent $t$-norms, then $\Sigma$ usually has to carry additional structure, such as that of a residuated lattice.
\end{remark}
\subsection{Relation to lattice-valued interval soft sets}
Theorem~\ref{thm main} collects the models that arise directly as special cases of SV-sets. Lattice-valued interval soft sets fit in the SV-set setting in a slightly different way: their connection with SV-sets is better understood through encoding, and in the usual membership-lattice setting by a natural interval-valued representation.

Zhang and Wang define a lattice-valued interval soft set (LVISS) as a pair $(F,A)$ with $A\subseteq E$ and $F:A\to I(L)$, where $I(L)$ is the interval lattice of a lattice $L$, with endpointwise joins and meets~\cite{ZhangWang2014LVISS}. They call $(F,A)$ \emph{simple} when every $F(e)$ is degenerate.

\begin{proposition}
Let $(F,A)$ be an LVISS based on a lattice $L$.
\begin{enumerate}[label=(\roman*)]
\item
Define
\[
\Phi_{\mathrm{for}}(F,A):A\times\{\ast\}\to I(L),\qquad
\Phi_{\mathrm{for}}(F,A)(e,\ast)=F(e).
\]
If $I(L)$ is also equipped with a suitable De Morgan negation, then this is an unparameterized SV-set.

\item
Assume that $L\cong V^{U}$ as lattices for some bounded lattice $V$, with pointwise order and pointwise joins/meets. Write
\[
F(e)=[f_e^{-},f_e^{+}]\in I(V^{U}),
\]
where $f_e^{-}\leq f_e^{+}$ pointwise, and define
\[
\Phi_{\mathrm{mem}}(F,A):U\times A\to I(V),\qquad
\Phi_{\mathrm{mem}}(F,A)(x,e)=[f_e^{-}(x),f_e^{+}(x)].
\]
Then $\Phi_{\mathrm{mem}}(F,A)$ is an interval-valued SV-set. On a fixed parameter domain, endpointwise LVISS unions and intersections agree with the pointwise SV joins and meets in the scale $I(V)$. If $V$ is a bounded De Morgan lattice, then $I(V)$ is also a bounded De Morgan lattice under
\[
\neg[a,b]=[\neg b,\neg a],
\]
and $\Phi_{\mathrm{mem}}$ respects complements as well.
\end{enumerate}
\end{proposition}

\begin{proof}
Part~(i) is simply another way of viewing the same information as a map into $I(L)$. For part~(ii), evaluation at $x\in U$ commutes with the pointwise lattice operations on $V^{U}$, and the interval operations in $I(V^{U})$ are defined endpointwise. Hence joins and meets are preserved on $U\times A$. If $V$ has a De Morgan negation, then the interval negation above gives the complement statement.
\end{proof}

\begin{remark}
\begin{enumerate}
\item
Let $\Sigma$ be a bounded lattice and let $A:U\times E\to\Sigma$ be an SV-set. For each
$e\in E$, the slice
$A_e:U\to\Sigma,\; A_e(x)=A(x,e),$
belongs to the function lattice $\Sigma^{U}$, so the assignment
$e\longmapsto [A_e,A_e]$
defines a simple LVISS on the full parameter domain $E$. Conversely, every simple LVISS of
this form arises from a unique SV-set. Thus SV-sets and simple LVISS are naturally related,
but only after passing to the function lattice.

\item
The membership-level construction in Proposition~3.4(ii) is natural only when the lattice
$L$ is presented as a function lattice $V^{U}$. If $L$ is treated as an abstract lattice, then
only the Proposition~3.4(i) is immediate. More generally, LVISS may
have non-degenerate intervals and variable domains $A\subseteq E$, so such features do not come from an ordinary pointwise
SV-set on a fixed parameter set.

\item
As a scale-based outer setting, SV-sets are broader, since the chosen scale need not be an
interval lattice. By contrast, LVISS are designed for the soft-set setting and retain the
variable parameter-domain structure $A\subseteq E$. In this sense, LVISS are more general
within interval-soft and soft-set theory, whereas SV-sets are more general at the level of
the value scale.

\item
For unions and intersections, only lattice structure is needed. Complements require extra
structure: if the underlying lattice carries a De Morgan negation, then the interval lattice
inherits the standard negation
$\neg[a,b]=[\neg b,\neg a].$
Without this, LVISS remain a lattice-theoretic framework.
\end{enumerate}
\end{remark}

\section{SV-Topological Spaces}\label{sec:topology}

  To define an SV-topology it is enough that the scale admits arbitrary joins and finite meets. The stronger hypothesis that $\Sigma$ is a complete chain is needed only when one passes from SV-open sets to ordinary topologies by strong cuts.

\begin{definition}
Let $\alpha\in \Sigma$ with $\alpha<1$, and let $A:U\to\Sigma$ be an SV-set with $E=\{\ast\}$. The strong $\alpha$-cut of $A$ is the ordinary subset
\[
A^{>\alpha}=\{x\in U:A(x)>\alpha\}.
\]
\end{definition}

\begin{theorem}\label{thm 4.2}
Let $\{A_i:i\in I\}$ be a family of SV-sets on $U$, let $A,B$ be SV-sets on $U$, and let $\alpha<1$. Then
\[
\left(\bigvee_{i\in I}A_i\right)^{>\alpha}=\bigcup_{i\in I}A_i^{>\alpha}.
\]
If, in addition, $\Sigma$ is a chain, then
\[
(A\wedge B)^{>\alpha}=A^{>\alpha}\cap B^{>\alpha}.
\]
\end{theorem}

\begin{proof}
Suppose $x\in \left(\bigvee_{i\in I}A_i\right)^{>\alpha}$. Then $\alpha<\bigvee_i A_i(x)$. If every $A_i(x)$ were at most $\alpha$, the join would also be at most $\alpha$, which is impossible. Hence some index $j$ satisfies $A_j(x)>\alpha$, so $x\in \bigcup_i A_i^{>\alpha}$. The converse inclusion is immediate.

Now assume that $\Sigma$ is a chain. If $x\in (A\wedge B)^{>\alpha}$, then $\alpha<A(x)\wedge B(x)$, hence $\alpha<A(x)$ and $\alpha<B(x)$, so $x\in A^{>\alpha}\cap B^{>\alpha}$. Conversely, if $x\in A^{>\alpha}\cap B^{>\alpha}$, then $\alpha<A(x)$ and $\alpha<B(x)$, and since $\Sigma$ is totally ordered this implies $\alpha<A(x)\wedge B(x)$. Thus $x\in (A\wedge B)^{>\alpha}$.
\end{proof}

\begin{definition}
Let $\Sigma$ be a complete lattice. An SV-topology on $U$ is a family $\tau\subseteq \SV(U,\{\ast\};\Sigma)$ containing the constant zero and constant one maps and closed under arbitrary pointwise joins and finite pointwise meets.

This is the direct scale-valued analogue of the usual closure axioms for fuzzy topological spaces~\cite{chang1968}.
\end{definition}

\begin{proposition}
If $\Sigma=\{0,1\}$ with the usual order, then SV-topologies on $U$ are exactly classical topologies on $U$ under characteristic functions.
\end{proposition}

\begin{proof}
An SV-set $A:U\to\{0,1\}$ is the characteristic function of a unique subset of $U$. With this correspondence, the constant zero and constant one maps become $\varnothing$ and $U$, arbitrary pointwise joins become unions, and finite pointwise meets become finite intersections. Therefore the SV-topology axioms are exactly the ordinary topology axioms in this case.
\end{proof}

\begin{theorem}\label{thm 4.5}
If $\tau$ is an SV-topology on $U$ and $\Sigma$ is a complete chain, then for each $\alpha<1$ the family
\[
\tau^{>\alpha}=\{A^{>\alpha}:A\in \tau\}
\]
is a classical topology on $U$.
\end{theorem}

\begin{proof}
Because $\alpha<1$, the strong cut of the constant zero map is empty and the strong cut of the constant one map is all of $U$. If $\{A_i:i\in I\}\subseteq \tau$, then
\[
\left(\bigvee_{i\in I}A_i\right)^{>\alpha}=\bigcup_{i\in I}A_i^{>\alpha}
\]
by Theorem~\ref{thm 4.2}, so $\tau^{>\alpha}$ is closed under arbitrary unions. If $A,B\in\tau$, then
\[
(A\wedge B)^{>\alpha}=A^{>\alpha}\cap B^{>\alpha},
\]
again by Theorem~\ref{thm 4.2}, so $\tau^{>\alpha}$ is closed under finite intersections.
\end{proof}

\begin{remark}
If the scale is a finite chain, weak cuts $A_{\alpha}=\{x\in U:\alpha\leq A(x)\}$ may also be used, because finite suprema are attained. The strong-cut version above is preferred in general complete chains because it avoids additional attainment assumptions.

The chain hypothesis in Theorem~\ref{thm 4.5} is essential. In the diamond lattice $M_3=\{0,p,q,1\}$ with $p$ and $q$ incomparable, let $A(x)=p$, $B(x)=q$, and $\alpha=0$ for some $x\in U$. Then $x\in A^{>0}\cap B^{>0}$, but $A(x)\wedge B(x)=0$, so $x\notin (A\wedge B)^{>0}$. Thus strong-cut families need not be closed under finite intersections when the scale is not a chain.
\end{remark}

\begin{definition}
Let $(U,\tau_U)$ and $(V,\tau_V)$ be SV-topological spaces over the same scale $\Sigma$. A map $f:U\to V$ is called SV-continuous if, for every $B\in \tau_V$, the pullback
\[
f^*B=B\circ f
\]
belongs to $\tau_U$.
\end{definition}

\begin{proposition}
Identity maps are SV-continuous, and compositions of SV-continuous maps are SV-continuous.
\end{proposition}

\begin{proof}
If $\mathrm{id}_U:U\to U$ is the identity map, then $\mathrm{id}_U^*A=A$ for every $A\in\tau_U$, so $\mathrm{id}_U$ is SV-continuous. If $f:U\to V$ and $g:V\to W$ are SV-continuous and $C\in\tau_W$, then
\[
(g\circ f)^*C=f^*(g^*C).
\]
Since $g^*C\in\tau_V$ and $f$ is SV-continuous, one gets $f^*(g^*C)\in\tau_U$. Hence $g\circ f$ is SV-continuous.
\end{proof}

\begin{proposition}
If $\Sigma$ is a complete chain and $f:(U,\tau_U)\to(V,\tau_V)$ is SV-continuous, then for each $\alpha<1$ the same set-function $f:U\to V$ is continuous from $\tau_U^{>\alpha}$ to $\tau_V^{>\alpha}$.
\end{proposition}

\begin{proof}
Let $B^{>\alpha}\in \tau_V^{>\alpha}$ with $B\in\tau_V$. Since $f$ is SV-continuous, $f^*B\in\tau_U$. Moreover,
\[
f^{-1}(B^{>\alpha})=(f^*B)^{>\alpha}.
\]
Therefore $f^{-1}(B^{>\alpha})\in \tau_U^{>\alpha}$, so $f$ is continuous between the induced crisp topologies.
\end{proof}

\begin{remark}
Intersections of SV-topologies on the same universe are again SV-topologies, since the defining axioms are preserved under intersection. Likewise, if $V\subseteq U$ and $\tau$ is an SV-topology on $U$, then restricting every $A\in\tau$ to $V$ gives the induced SV-topology on $V$.

If the parameter set $E$ is arbitrary, the same axioms define a parameterized SV-topology
$\tau\subseteq \SV(U,E;\Sigma)$. This may be viewed as a scale-valued soft topology on $U$:
for each $e\in E$, the slice family $\tau_e=\{A_e:A\in\tau\}$ is an ordinary SV-topology on $U$,
and if $\Sigma$ is a complete chain, then for each $\alpha<1$ the parameterwise strong cuts
$A^{>\alpha}(e)=\{x\in U:A(x,e)>\alpha\}$
form a soft topology on $U$ with parameter set $E$. In particular, $\Sigma=\{0,1\}$ gives soft topologies, while $\Sigma=[0,1]$ gives fuzzy soft topologies.
\end{remark}
\section{One Algebraic Extension: SV-Subgroups}\label{sec:sv-subgroups}

The SV-set language also extends naturally to algebraic structures. In this section only the top element and the meet of the scale are used, so the subgroup theory is formally lighter than the full bounded De Morgan setting. For simplicity we again work in the unparameterized case $E=\{\ast\}$.

\begin{definition}
Let $G$ be a group with identity $e_G$. An SV-set $A:G\to\Sigma$ is called an SV-subgroup if $A(e_G)=1$ and
$A(xy^{-1})\geq A(x)\wedge A(y) \quad \text{for all } x,y\in G.$\\
When $\Sigma=\{0,1\}$, this is exactly the characteristic function of an ordinary subgroup. When $\Sigma=[0,1]$, it reduces to the familiar fuzzy subgroup condition~\cite{rosenfeld1971}.
\end{definition}

\begin{lemma}\label{lem 5.2}
If $A$ is an SV-subgroup of $G$, then for all $x,y\in G$ one has
\[
A(x)\leq A(e_G),
\qquad
A(x^{-1})=A(x),
\qquad
A(xy)\geq A(x)\wedge A(y).
\]
\end{lemma}

\begin{proof}
Taking $y=x$ in the definition we get
\[
A(e_G)=A(xx^{-1})\geq A(x)\wedge A(x)=A(x),
\]
so $A(x)\leq A(e_G)$ for all $x$. Taking $x=e_G$ gives
\[
A(y^{-1})=A(e_Gy^{-1})\geq A(e_G)\wedge A(y)=A(y),
\]
and replacing $y$ by $y^{-1}$ gives the reverse inequality, hence $A(y^{-1})=A(y)$. Finally,
\[
A(xy)=A\bigl(x(y^{-1})^{-1}\bigr)\geq A(x)\wedge A(y^{-1})=A(x)\wedge A(y).
\]
\end{proof}

\begin{proposition}
For a map $A:G\to\Sigma$, the following are equivalent.
\begin{enumerate}[label=(\roman*)]
    \item $A(e_G)=1$ and $A(xy^{-1})\geq A(x)\wedge A(y)$ for all $x,y\in G$.
    \item $A(e_G)=1$, $A(x^{-1})=A(x)$ for all $x\in G$, and $A(xy)\geq A(x)\wedge A(y)$ for all $x,y\in G$.
\end{enumerate}
\end{proposition}

\begin{proof}
The implication (i)$\Rightarrow$(ii) is Lemma~\ref{lem 5.2}. Conversely, if (ii) holds, then
\[
A(xy^{-1})\geq A(x)\wedge A(y^{-1})=A(x)\wedge A(y),
\]
so (i) follows.
\end{proof}

\begin{theorem}
Assume that $\Sigma$ admits arbitrary meets. If $\{A_i:i\in I\}$ is a family of SV-subgroups of $G$, then the pointwise meet
\[
A=\bigwedge_{i\in I}A_i
\]
is an SV-subgroup of $G$.
\end{theorem}

\begin{proof}
At the identity one has
\[
A(e_G)=\bigwedge_{i\in I}A_i(e_G)=\bigwedge_{i\in I}1=1.
\]
Also, for all $x,y\in G$,
\[
\begin{aligned}
A(xy^{-1})
&=\bigwedge_{i\in I}A_i(xy^{-1})
\geq \bigwedge_{i\in I}(A_i(x)\wedge A_i(y))\\
&=\left(\bigwedge_{i\in I}A_i(x)\right)\wedge\left(\bigwedge_{i\in I}A_i(y)\right)
=A(x)\wedge A(y).
\end{aligned}
\]
Hence $A$ is an SV-subgroup.
\end{proof}

\begin{theorem}
For $\alpha\in\Sigma$, define the weak level set
\[
A_{\alpha}=\{x\in G:\alpha\leq A(x)\}.
\]
Then an SV-set $A:G\to\Sigma$ is an SV-subgroup if and only if $A_{\alpha}$ is a subgroup of $G$ for every $\alpha\in\Sigma$.
\end{theorem}

\begin{proof}
Assume first that $A$ is an SV-subgroup and fix $\alpha\in\Sigma$. Since $A(e_G)=1$, one has $e_G\in A_{\alpha}$. If $x,y\in A_{\alpha}$, then $\alpha\leq A(x)$ and $\alpha\leq A(y)$, so
\[
A(xy^{-1})\geq A(x)\wedge A(y)\geq \alpha.
\]
Thus $xy^{-1}\in A_{\alpha}$, and $A_{\alpha}$ is a subgroup.

Conversely, assume that every $A_{\alpha}$ is a subgroup. Taking $\alpha=1$ shows that $A_1$ is nonempty, hence $e_G\in A_1$ and therefore $A(e_G)=1$. Now fix $x,y\in G$ and put $\alpha=A(x)\wedge A(y)$. Then $x,y\in A_{\alpha}$, so $xy^{-1}\in A_{\alpha}$. Hence
\[
A(xy^{-1})\geq \alpha=A(x)\wedge A(y),
\]
and $A$ is an SV-subgroup.
\end{proof}

\begin{theorem}
Let $\varphi:H\to G$ be a group homomorphism. If $A$ is an SV-subgroup of $G$, then $\varphi^*A=A\circ \varphi$ is an SV-subgroup of $H$.
\end{theorem}

\begin{proof}
Since $\varphi(e_H)=e_G$, one gets $\varphi^*A(e_H)=A(e_G)=1$. Moreover, for $x,y\in H$,
\[
\varphi^*A(xy^{-1})=A\bigl(\varphi(xy^{-1})\bigr)=A\bigl(\varphi(x)\varphi(y)^{-1}\bigr)\geq A\bigl(\varphi(x)\bigr)\wedge A\bigl(\varphi(y)\bigr)=\varphi^*A(x)\wedge \varphi^*A(y).
\]
Hence $\varphi^*A$ is an SV-subgroup of $H$.
\end{proof}

\begin{remark}
If $\Sigma$ is complete and $\varphi:H\to G$ is a homomorphism, one may also consider the image-type construction given by $\varphi_!A$ from Definition~\ref{def 2.8}. Whether $\varphi_!A$ is again an SV-subgroup depends on additional distributivity properties of the scale, so we do not pursue that question here.

If the parameter set $E$ is arbitrary, the same definition gives a parameterized SV-subgroup
$A:G\times E\to\Sigma.$
This may be viewed as a scale-valued soft subgroup on $G$: for each $e\in E$, the slice
$A_e:G\to\Sigma$, defined by $A_e(x)=A(x,e)$, is an SV-subgroup of $G$. For each $\alpha\in\Sigma$, the parameterwise weak cuts
$A_{\alpha}(e)=\{x\in G:\alpha\leq A(x,e)\}$
then form a soft family of ordinary subgroups of $G$. In particular, $\Sigma=\{0,1\}$ gives
soft families of ordinary subgroups, while $\Sigma=[0,1]$ gives fuzzy soft subgroups.
\end{remark}

\section{Applied SV-Set Models with Product Scales}\label{sec:decision}
One of the most useful features of the SV setting is that the scale can store more than one kind of local information at the same time. This is especially natural in decision problems, where a grade is often accompanied by confidence, support, or evidence. Fuzzy soft sets already give a parameterized way to organize graded evaluations~\cite{roymaji2007}, while the supplier-selection literature makes it clear that confidence or credibility is different from the hesitancy represented in intuitionistic fuzzy information~\cite{aggarwal2019}. Product scales offer a simple way to keep both kinds of information visible until the final stage of decision making.

If $\Sigma_1$ and $\Sigma_2$ are scales, then the product $\Sigma_1\times \Sigma_2$ becomes a scale under the componentwise order, componentwise joins and meets, and componentwise involution. We use this product scale $\Sigma_1\times \Sigma_2$ in next examples.

\subsection{A product-scale model for laptop selection}

\begin{example}
A company wants to adopt a single laptop model for a group of employees. For each candidate it records both a suitability grade and the amount of supporting evidence behind that grade, such as benchmark data, battery tests, or service records.

Let
\[
U=\{L_1,L_2,L_3,L_4\},
\qquad
E=\{\textsf{performance},\textsf{battery},\textsf{portability},\textsf{service}\},
\]
and take the product scale
\[
\Sigma=[0,1]\times\{0,1,\dots,10\}.
\]
Suppose $(\mu,m)$ is a member of product scale $\Sigma$ and interpret it as ``grade $\mu$ supported by $m$ evidence items ''. 
Define the SV-set $A:U\times E\to\Sigma$ by Table~\ref{tab:laptop-data}.
\begin{table}[h]
\centering
\small
\caption{Laptop evaluations with grade and evidence count.}
\label{tab:laptop-data}
\begin{tabular}{@{}lcccc@{}}
\toprule
 & \textsf{performance} & \textsf{battery} & \textsf{portability} & \textsf{service}\\
\midrule
$L_1$ & $(0.90,8)$ & $(0.60,6)$ & $(0.70,8)$ & $(0.50,3)$ \\
$L_2$ & $(0.80,9)$ & $(0.80,7)$ & $(0.60,6)$ & $(0.70,5)$ \\
$L_3$ & $(0.70,5)$ & $(0.90,4)$ & $(0.80,5)$ & $(0.60,4)$ \\
$L_4$ & $(0.60,10)$ & $(0.70,10)$ & $(0.90,9)$ & $(0.80,9)$ \\
\bottomrule
\end{tabular}
\end{table}

Using the conservative rule
\[
B(L)=\bigwedge_{e\in E}A(L,e),
\]
one obtains
\[
B(L_1)=(0.50,3),\qquad B(L_2)=(0.60,5),\qquad B(L_3)=(0.60,4),\qquad B(L_4)=(0.60,9).
\]
Consider the expression
\[
r_\lambda(\mu,m)=\lambda\mu+(1-\lambda)\frac{m}{10},
\]
Here $\lambda\in(0,1)$ tells us how much importance is given to grade and how much to evidence: a larger $\lambda$ gives more weight to grade, while a smaller $\lambda$ gives more weight to evidence. With $\lambda=0.7$, the resulting values are
\[
r_{0.7}(B(L_1))=0.44,\quad r_{0.7}(B(L_2))=0.57,\quad r_{0.7}(B(L_3))=0.54,\quad r_{0.7}(B(L_4))=0.69.
\]
Hence
\[
L_4\succ L_2\succ L_3\succ L_1.
\]
In this example the ranking stays the same for every $0<\lambda<1$, so the choice $\lambda=0.7$ is only for illustration. If one uses only the grade coordinate, then $L_2$, $L_3$, and $L_4$ all tie at grade $0.60$, whereas the full SV-profile separates them by using the evidence coordinate.
\end{example}
\subsection{A supplier-selection model with audit evidence}

\begin{example}
Supplier selection is a common multi-criteria decision problem, and in practice the recorded evaluations often come with different levels of support or credibility~\cite{aggarwal2019}. The SV setting allows both aspects to be kept together in a single model.

Let
\[
U=\{S_1,S_2,S_3,S_4\},
\qquad
E=\{\textsf{cost},\textsf{quality},\textsf{delivery},\textsf{sustainability},\textsf{flexibility}\},
\]
and take
\[
\Sigma=[0,1]\times\{0,1,\dots,5\}.
\]
Here $(\mu,m)$ means ``suitability grade $\mu$ supported by $m$ independent audits, tests, or documents.'' Define $A:U\times E\to\Sigma$ by Table~\ref{tab:supplier-data}.
\begin{table}[h]
\centering
\small
\caption{Supplier evaluations with grade and evidence count.}
\label{tab:supplier-data}
\begin{tabular}{@{}lccccc@{}}
\toprule
 & \textsf{cost} & \textsf{quality} & \textsf{delivery} & \textsf{sustainability} & \textsf{flexibility}\\
\midrule
$S_1$ & $(0.80,3)$ & $(0.70,4)$ & $(0.60,2)$ & $(0.90,1)$ & $(0.60,2)$ \\
$S_2$ & $(0.70,5)$ & $(0.90,4)$ & $(0.70,4)$ & $(0.60,2)$ & $(0.70,3)$ \\
$S_3$ & $(0.60,4)$ & $(0.80,3)$ & $(0.90,5)$ & $(0.70,4)$ & $(0.80,4)$ \\
$S_4$ & $(0.75,2)$ & $(0.75,2)$ & $(0.75,2)$ & $(0.75,2)$ & $(0.75,2)$ \\
\bottomrule
\end{tabular}
\end{table}

As in Example 7.1  with 
\[
B(S)=\bigwedge_{e\in E}A(S,e),
\]
one gets
\[
B(S_1)=(0.60,1),\qquad B(S_2)=(0.60,2),\qquad B(S_3)=(0.60,3),\qquad B(S_4)=(0.75,2).
\]
Now use the expression
\[
r_\lambda(\mu,m)=\lambda\mu+(1-\lambda)\frac{m}{5}.
\]
For $\lambda=0.7$,
\[
r_{0.7}(B(S_1))=0.48,\quad r_{0.7}(B(S_2))=0.54,\quad r_{0.7}(B(S_3))=0.60,\quad r_{0.7}(B(S_4))=0.645,
\]
so
\[
S_4\succ S_3\succ S_2\succ S_1.
\]
For $\lambda=0.4$,
\[
r_{0.4}(B(S_1))=0.36,\quad r_{0.4}(B(S_2))=0.48,\quad r_{0.4}(B(S_3))=0.60,\quad r_{0.4}(B(S_4))=0.54,
\]
so
\[
S_3\succ S_4\succ S_2\succ S_1.
\]
The break-even value between $S_3$ and $S_4$ is obtained from
\[
\lambda\cdot 0.60+(1-\lambda)\cdot \frac35=\lambda\cdot 0.75+(1-\lambda)\cdot \frac25,
\]
which gives $\lambda=\frac47$. Thus the SV-model not only keeps grade and evidence together, but also makes the balance between them explicit.
\end{example}

\subsection{A hybrid ranking result}

If a product-scale value is reduced to only one coordinate, then any difference carried by the other coordinate is lost. In the present setting, a grade-only reduction cannot distinguish alternatives that differ only in evidence, while an evidence-only reduction cannot distinguish alternatives that differ only in grade. The next theorem makes this clear in the case of three alternatives.

\begin{theorem}
Let $k\geq 1$ and let $B:U\to [0,1]\times\{0,1,\dots,k\}$ be the aggregated SV-set. Assume there exist $u_1,u_2,u_3\in U$ such that
\[
B(u_1)=(\mu,m),
\qquad
B(u_2)=(\mu,m'),
\qquad
B(u_3)=(\mu',m),
\]
with $\mu<\mu'$ and $m<m'$. Define
\[
r_\lambda(\alpha,n)=\lambda\alpha+(1-\lambda)\frac{n}{k},
\qquad 0<\lambda<1,
\]
and set
\[
\lambda^\ast=\frac{m'-m}{k(\mu'-\mu)+(m'-m)}\in(0,1).
\]
Then the following hold.
\begin{enumerate}[label=(\roman*)]
    \item Any scoring, ranking, or choice procedure depending only on the grade projection $\pi_1\circ B$ cannot distinguish $u_1$ from $u_2$.
    \item Any scoring, ranking, or choice procedure depending only on the evidence projection $\pi_2\circ B$ cannot distinguish $u_1$ from $u_3$.
    \item If $0<\lambda<\lambda^\ast$, then
    \[
    r_\lambda(B(u_2))>r_\lambda(B(u_3))>r_\lambda(B(u_1)).
    \]
    \item If $\lambda^\ast<\lambda<1$, then
    \[
    r_\lambda(B(u_3))>r_\lambda(B(u_2))>r_\lambda(B(u_1)).
    \]
\end{enumerate}
Consequently, for every $\lambda\in(0,1)\setminus\{\lambda^\ast\}$, the hybrid score gives a strict total order on $\{u_1,u_2,u_3\}$, whereas neither one-coordinate reduction can produce the same order without ties.
\end{theorem}

\begin{proof}
For~(i), one has $\pi_1(B(u_1))=\pi_1(B(u_2))=\mu$, so any rule factoring through $\pi_1$ sees identical data for $u_1$ and $u_2$ and therefore cannot distinguish them. For~(ii), one has $\pi_2(B(u_1))=\pi_2(B(u_3))=m$, so any rule factoring through $\pi_2$ must identify $u_1$ and $u_3$.

Next,
\[
r_\lambda(B(u_2))-r_\lambda(B(u_1))=(1-\lambda)\frac{m'-m}{k}>0,
\]
and
\[
r_\lambda(B(u_3))-r_\lambda(B(u_1))=\lambda(\mu'-\mu)>0,
\]
for every $0<\lambda<1$. Hence both $u_2$ and $u_3$ are ranked strictly above $u_1$.

To compare $u_2$ and $u_3$, observe that
\[
r_\lambda(B(u_3))-r_\lambda(B(u_2))=\lambda(\mu'-\mu)-(1-\lambda)\frac{m'-m}{k}.
\]
This quantity is zero exactly when
\[
\lambda=\frac{m'-m}{k(\mu'-\mu)+(m'-m)}=\lambda^\ast.
\]
It is therefore negative for $0<\lambda<\lambda^\ast$ and positive for $\lambda^\ast<\lambda<1$. This gives~(iii) and~(iv). Since a grade-only rule must tie $u_1$ with $u_2$, and an evidence-only rule must tie $u_1$ with $u_3$, neither can realize the same strict total order.
\end{proof}

\begin{example}\label{ex:proposal-order}
Let $k=5$ and suppose three procurement proposals have aggregated profiles
\[
B(P_1)=(0.65,2),\qquad B(P_2)=(0.65,4),\qquad B(P_3)=(0.80,2).
\]
Then
\[
\lambda^\ast=\frac{4-2}{5(0.80-0.65)+(4-2)}=\frac{8}{11}.
\]
If $\lambda=\tfrac12<\lambda^\ast$, then
\[
r_{1/2}(B(P_1))=0.525,\qquad r_{1/2}(B(P_2))=0.725,\qquad r_{1/2}(B(P_3))=0.60,
\]
so
\[
P_2\succ P_3\succ P_1.
\]
If $\lambda=0.9>\lambda^\ast$, then
\[
r_{0.9}(B(P_1))=0.625,\qquad r_{0.9}(B(P_2))=0.665,\qquad r_{0.9}(B(P_3))=0.76,
\]
so
\[
P_3\succ P_2\succ P_1.
\]
In both cases, the grade-only projection ties $P_1$ with $P_2$, while the evidence-only projection ties $P_1$ with $P_3$.
\end{example}

\section{Conclusion}

This paper studies generalized membership models through one simple idea: a map into a chosen scale. When the scale is a bounded De Morgan lattice, many familiar models can be handled in the same framework. In this setting, the usual pointwise operations, transport maps, and cut constructions are also available.

The main strength of this viewpoint is that the outer form remains the same, while the meaning of the values can change from one model to another. This is seen in the product-scale applications and also in the discussion of Type-2 fuzzy sets and lattice-valued interval soft sets. At the same time, writing a model in SV form does not mean that all such models become identical. Some of them still depend on extra structure outside the stored values.

There is still good scope for further work, especially on richer scales, more algebraic and topological examples, and more application-based constructions. Even in its present form, the SV viewpoint gives a simple and workable way to organize a wide range of generalized set models.

\section*{Acknowledgments}
The author has no acknowledgments to report.

\section*{Funding}
The author received no financial support for the research, authorship, and/or publication of this article.

\section*{Disclosure of interest}
The author reports there are no competing interests to declare.

\section*{Data availability statement}
No new data were generated or analysed in this study. Data sharing is not applicable to this article.


\end{document}